\documentclass[10pt,a4paper,oneside,english]{amsart}
\usepackage[T1]{fontenc}
\usepackage[latin9]{inputenc}
\usepackage{units}
\usepackage{amsthm}
\usepackage{amssymb}
\usepackage{bbm}

\makeatletter

\numberwithin{equation}{section}
\numberwithin{figure}{section}
\theoremstyle{plain}
\newtheorem{thm}{\protect\theoremname}[section]
  \theoremstyle{plain}
  \newtheorem{lem}[thm]{\protect\lemmaname}
  \theoremstyle{plain}
  
  \theoremstyle{remark}
  \newtheorem{rem}[thm]{\protect\remarkname}
  \theoremstyle{plain}

\@ifundefined{date}{}{\date{}}

\usepackage{amsthm}

\usepackage{cite}

\newtheorem{theorem}{Theorem}[section]

\newtheorem{example}[theorem]{Example}

\providecommand{\lemmaname}{Lemma}
  \providecommand{\propositionname}{Proposition}
  \providecommand{\remarkname}{Remark}
\providecommand{\theoremname}{Theorem}

\usepackage{babel}
\providecommand{\lemmaname}{Lemma}
  \providecommand{\propositionname}{Proposition}
  \providecommand{\remarkname}{Remark}
\providecommand{\theoremname}{Theorem}

\usepackage{babel}
\providecommand{\corollaryname}{Corollary}
  \providecommand{\lemmaname}{Lemma}
  \providecommand{\propositionname}{Proposition}
  \providecommand{\remarkname}{Remark}
\providecommand{\theoremname}{Theorem}

\makeatother

\usepackage{babel}
  \providecommand{\corollaryname}{Corollary}
  \providecommand{\lemmaname}{Lemma}
  \providecommand{\propositionname}{Proposition}
  \providecommand{\remarkname}{Remark}
\providecommand{\theoremname}{Theorem}

\begin{document}

\title [Weak$^*$
fixed point property and the space of affine functions]  {Weak$^*$
	fixed point property \\ and the space of affine functions}

\author{E. Casini}

\address{Dipartimento di Scienza e Alta Tecnologia, Università dell'Insubria,
via Valleggio 11, 22100 Como, Italy }

\email{emanuele.casini@uninsunbria.it}

\author{E. Miglierina}

\address{Dipartimento di Matematica per le Scienze economiche, finanziarie ed attuariali,
Università Cattolica del Sacro Cuore, Via Necchi 9, 20123 Milano,
Italy }

\email{enrico.miglierina@unicatt.it}

\author{\L. Piasecki}

\address{Instytut Matematyki,
Uniwersytet Marii Curie-Sk{\l}odowskiej, Pl. Marii Curie-Sk{\l}odowskiej 1, 20-031 Lublin, Poland}

\email{piasecki@hektor.umcs.lublin.pl}
\begin{abstract}
First we prove that if a separable Banach space $X$ contains an isometric copy of an infinite-dimensional space $A(S)$ of affine continuous functions on a Choquet simplex $S$, then its dual $X^*$ lacks the weak$^*$ fixed point property for nonexpansive mappings.
Then, we show that the dual of a separable Lindenstrauss space $X$ fails the weak$^*$ fixed point property for nonexpansive mappings if and only if $X$ has a quotient isometric to some space $A(S)$. Moreover, we provide an example showing that \textquotedblleft quotient\textquotedblright{} cannot be replaced by \textquotedblleft subspace\textquotedblright{}. Finally, it is worth to be mentioned that in our characterization the space $A(S)$ cannot be substituted by any space $\mathcal{C}(K)$ of continuous functions on a compact Hausdorff $K$.
\end{abstract}



\maketitle

\section{Introduction}
Let $X$ be an infinite-dimensional real Banach space and let us denote by $B_X$ its closed unit ball and by $S_{\! X}$ its unit sphere. A Banach space $X$ is called an $L_1$-predual  (or a Lindenstrauss space) if its dual $X^*$ is isometric to $L_1(\mu)$ for some
measure $\mu$. The most widely studied $L_1$-preduals are classical  Banach spaces $\mathcal{C}(K)$ of continuous functions on a compact Hausdorff space $K$. In this paper two other subclasses of Lindenstrauss spaces play a crucial role. The first one is the well-known class of spaces $A(S)$ of continuous affine functions defined on a Choquet simplex $S$. It is worth to be mentioned that the class of $A(S)$ spaces is strictly wider than the class of $\mathcal{C}(K)$ spaces. Moreover, an $L_1$-predual $X$ is an $A(S)$ space if and only if $B_X$ has an extreme point (\cite{Semadeni1964}). The second class that we are interested in is the collection of all hyperplanes in $c$, the space of convergent sequences endowed with the standard supremum norm. These hyperplanes were extensively studied in \cite{Casini-Miglierina-Piasecki2014,Casini-Miglierina-Piasecki2015}. Here we recall some notations and properties about these spaces that will be useful in the sequel. It is well-known that $c^*$ can be
isometrically identified with $\ell_1$ in the following way: for
every $x^*\in c^*$ there exists a unique element $f=(f(1),f(2),\dots)\in \ell_1$
such that
\[
x^*(x)=\sum_{n=0}^\infty f(n+1)x(n)=f(x)
\]
with $x=(x(1),x(2),\dots)\in c$ and $x(0)=\lim x(n)$.
Let $f\in S_{\! c^*}$. Consider the hyperplane in $c$ defined by
\[
W_f=\left\lbrace x\in c:f(x)=0 \right\rbrace.
\]
In \cite{Casini-Miglierina-Piasecki2014}, the following results are proved:
\begin{enumerate}
	\item $W_f^*$ is isometric to $\ell_1$ if and only if there exists $j_0\geq1$ such that $|f(j_0)|\geq 1/2$,
	\item $W_f$ is isometric to $c$ if and only if there exists $j_0\geq2$ such that $|f(j_0)|\geq 1/2$.
\end{enumerate}

\quad

Under the additional assumption $|f(1)|\geq 1/2$ and $|f(j)|< 1/2$ for every $j\geq2$, Theorem 4.3 in \cite{Casini-Miglierina-Piasecki2014} identifies $W_f^*$ and $\ell_1$ by giving the following dual action: for every $x^*\in W_f^*$ there exists a unique element $g\in\ell_1$ such that
\begin{equation}\label{dualityW}
x^*(x)=\sum_{n=1}^\infty g(n)x(n)=g(x),
\end{equation}
where $x=(x(1),x(2),\dots)\in W_f$. Moreover, if $\left\lbrace e^*_n\right\rbrace$ denotes the standard basis in $\ell_1$, then 
\begin{equation}
e^*_{n} \xrightarrow{\sigma(\ell_{1},W_{f})} \left( -\frac{f(2)}{f(1)},-\frac{f(3)}{f(1)},-\frac{f(4)}{f(1)},\dots\right), \tag{$\heartsuit$}
\end{equation}
where $\sigma(X^*,X)$ denotes the weak$^*$ topology on $X^*$ induced by $X$.

The aim of this paper is to investigate the relationships between presence of an isometric copy of an $A(S)$ space in a separable space $X$ and the failure of weak$^*$ fixed point property for nonexpansive mapping in the dual space $X^*$. We recall that the space $X^*$ is said to have the weak$^*$ fixed
point property (briefly, $\sigma(X^*,X)$-FPP) if for every nonempty, convex,
$\sigma(X^*,X)$-compact subset $C$ of $X^*$, every nonexpansive mapping (i.e., a mapping $T:C\rightarrow C$ such that $\|T(x)-T(y)\|\leq \|x-y\|$ for all $x,y\in C$) has a fixed point.

First we prove that if a separable Banach space $X$ contains an isometric copy of an infinite-dimensional space $A(S)$, then its dual $X^*$ lacks the $\sigma(X^*,X)$-FPP (see Theorem \ref{thm sufficient}). This sufficient condition can be extended by considering a quotient containing an isometric copy of an $A(S)$ space (see Remark \ref{rem quotient}). Our theorem extends, in a separable case, the result by Smyth stating  that $\mathcal{C}(K)^*$ fails the $\sigma\left(\mathcal{C}(K)^*,\mathcal{C}(K) \right)$-FPP (\cite{Smyth1994}).  

In the last section we discuss the case of separable $L_1$-preduals and we completely characterize the weak$^*$ topologies that fail the $\sigma(X^*,X)$-FPP. Indeed, we prove that the dual $X^*$ of a separable Lindenstrauss space $X$ fails the $\sigma(X^*,X)$-FPP if and only if $X$ has a quotient isometric to an infinite-dimensional $A(S)$ space for some Choquet simplex $S$. We also show that the latter condition may be replaced by: $X$ has a quotient containing an isometric copy of an infinite-dimensional $A(S)$ space for some Choquet simplex $S$ (see Theorem \ref{thm separable Lindenstrauss}). 
Finally, one may ask whether in the previous results the space $A(S)$ can be replaced by a space $\mathcal{C}(K)$ or if the quotient can be removed, in a sense that these conditions can be replaced by: $X$ has a subspace isometric to an infinite-dimensional $A(S)$ space. Remark \ref{Remark C(K)} and Example \ref{Example Hyperplane} show that the answers for both questions are negative.

\section{Weak$^*$ fixed point property in the dual of separable Banach space }\label{Section separable}

This section is devoted to study a sufficient condition for the failure of the $\sigma(X^*,X)$-FPP for a generic separable space $X$ in term of the presence of an isometric copy of an $A(S)$ space. In \cite{Casini-Miglierina-Piasecki2015}, we provided a sufficient condition of similar type but based on the presence of an isometric copy of the so-called bad hyperplane in $c$. Recall that the hyperplane $W_f$ is called {\it bad} if $f\in\ell_1$ is such that $\left\| f\right\|=1$,
$\left| f(1)\right|=\frac{1}{2}$ and the set $N^+=\left\lbrace n\in
\mathbb{N}:f(1) f(n+1) \leq 0\right\rbrace $ is infinite. The word bad was chosen with respect to the $\sigma(\ell_1,W_f)$-FPP since the dual of every bad $W_f$ lacks the $\sigma(\ell_1,W_f)$-FPP. The statement of our aforementioned result is: 
\begin{thm}\label{Theorem bad}(Theorem 3.7 in \cite{Casini-Miglierina-Piasecki2015})
	Let $X$ be a separable Banach space. If $X$ contains a subspace isometric to a bad hyperplane, then $X^*$ fails the $\sigma(X^*,X)$-FPP.  
\end{thm}

The main result of this section aims to provide a sufficient condition for the failure of the $\sigma(X^*,X)$-FPP by reformulating Theorem \ref{Theorem bad}, where bad hyperplane is replaced by an $A(S)$ space. In order to prove it, we need to recall a known result that we quote here for the sake of convenience of the reader. 

\begin{lem}\label{lemmaDS} (\cite{Dunford Schwartz}, p. 441)
	Let $X$ be a closed subspace of $\mathcal{C}(K)$ of all real continuous functions on a compact Hausdorff space $K$. For each $q \in K$ let $x_q^*\in X^*$ be defined by
	$$
	x^*_q(f)=f(q),\quad f\in X.
	$$
	Then every extreme point of the closed unit sphere of $X^*$ is of the form $\pm x^*_q$ with $q \in K$.
\end{lem}

\begin{thm}\label{thm sufficient}
	Let $X$ be a separable Banach space. If $X$ contains an isometric copy of some infinite-dimensional $A(S)$ space, then $X^*$ lacks the $\sigma(X^*,X)$-FPP. 
\end{thm}
\begin{proof}
	We start the proof by considering the case when $ A(S)^*$ is nonseparable and hence $X^*$ is nonseparable. By Theorem 2.3 in \cite{Lazar-Lindenstrauss1971} every separable Lindenstrauss space $X$ with nonseparable dual contains a subspace isometric to the space $\mathcal{C}(\Delta)$, where $\Delta$ is the Cantor set. Since $\mathcal{C}(\Delta)$ contains an isometric copy of $c$, by Theorem \ref{Theorem bad}, $X^*$ lacks the $\sigma(X^*,X)$-FPP. 
	
	Therefore, it remains to consider the case when $A(S)^*=\ell_1$. 
	Let $\mathbbm{1}$ denote the identically equal $1$ function defined on $S$. Let $\pi$ denote the canonical embedding of $A(S)$ into $ A(S)^{**}=\ell_{\infty}$. Since $A(S)\subset \mathcal{C}(S)$, by Lemma \ref{lemmaDS} there exists a sequence of signs $(\varepsilon(n))_{n\in \mathbb{N}}$, $\varepsilon(n) = \pm 1$ for all $n\in \mathbb{N}$ such that $\pi(\mathbbm{1})=(\varepsilon(1),\varepsilon(2),\varepsilon(3),\dots)$. Without loss of generality, we may assume that there is a subsequence $\left\lbrace \varepsilon(n_j)\right\rbrace_{j \in \mathbb{N}}$ of $\left\lbrace \varepsilon(n)\right\rbrace_{n \in \mathbb{N}}$ such that $\varepsilon(n_j)=1$ for each $j \in \mathbb{N}$. 
	Let us consider the set 
	$$
	C=\left\lbrace f\in \ell_1: \pi(\mathbbm{1}) (f)=1\right\rbrace \cap B_{\ell_1}= \left\lbrace f=\left( f(1),f(2),\dots\right) \in B_{\ell_1}: \sum_{i=1}^{\infty}\varepsilon(i)f(i)=1\right\rbrace. 
	$$
	It is easy to see that $C$ is a nonempty, convex and $\sigma(\ell_1,A(S))$-compact subset of $\ell_1$. Now, by choosing a subsequence, we may assume that $\left\lbrace e^*_{n_j}\right\rbrace_{j \in \mathbb{N}}$ is $\sigma(\ell_1,A(S))$-convergent to some $e^*$. Since $e^* \in C$, we have $\left\| e^*\right\|=1$ and $e^*(n_j)\geq 0$ for every $j \in \mathbb{N}$. 
	
	From now on, the proof follows the approach already used in the proof of Theorem 3.7 in \cite{Casini-Miglierina-Piasecki2015}. For the convenience of the reader we repeat here the relevant part of that proof.	
	Again, by choosing a subsequence, we may assume that 	
	$w_{0}=e^{*}-u_{0}\neq0$ where
	$u_{0}=\sum_{j=1}^{+\infty}e^{*}(n_{j})e_{n_{j}}^{*}$. Let $x_{n_j}^*$ be a norm-preserving extension of $e_{n_j}^*$ to the whole space $X$. Now we
	consider the extension of $u_{0}$ to the whole space $X$ defined by
	$\widetilde{{u}_{0}}=\sum_{j=1}^{\infty}e^{*}(n_{j})x_{n_{j}}^{*}$
	and the elements $\widetilde{{w}_{0}}=x^{*}-\widetilde{{u}_{0}}$ and
	$\widetilde{w}=\frac{\widetilde{{w}_{0}}}{\left\Vert
		w_{0}\right\Vert }$. Now, by adapting to our framework the approach
	developed in the last part of the proof of Theorem 8 in
	\cite{Japon-Prus2004}, we show that the $\sigma
	(X^*,X)$-compact, convex set
	\[
	D=\left\{
	\mu_{1}x^{*}+\mu_{2}\widetilde{w}+\sum_{j=1}^{\infty}\mu_{j+2}x_{n_{j}}^{*}:\,\sum_{k=1}^{\infty}\mu_{k}=1,\,\mu_{k}\geq0,\,
	k=1,2,\dots\right\}
	\]
	is equal to
	\[
	D=\left\{
	\lambda_{1}\widetilde{w}+\sum_{j=1}^{\infty}\lambda_{j+1}x_{n_{j}}^{*}:\,\sum_{k=1}^{\infty}\lambda_{k}=1,\,\lambda_{k}\geq0,\,
	k=1,2,\dots\right\} .
	\]
	Next, we consider the map $T:D\rightarrow D$ defined by:
	\[
	T\left(\lambda_{1}\widetilde{w}+\sum_{j=1}^{\infty}\lambda_{j+1}x_{n_{j}}^{*}\right)=\sum_{j=1}^{\infty}\lambda_{j}x_{n_{j}}^{*}.
	\]
	Since
	$x=\lambda_{1}\widetilde{w}+\sum_{j=1}^{\infty}\lambda_{j+1}x_{n_{j}}^{*}\in
	D$ has a unique representation, the map $T$ is well defined. Moreover
	it is a nonexpansive map. Indeed, for every
	$\alpha=(\alpha_{1},\alpha_{2},\dots)$ such that $\sum_{j=1}^{\infty}\left|  \alpha_j \right|< \infty$, it holds
	\[
	\left\Vert
	\alpha_{1}\widetilde{w}+\sum_{j=1}^{\infty}\alpha_{j+1}x_{n_{j}}^{*}\right\Vert
	\geq\left\Vert \alpha_{1}\frac{w_{0}}{\left\Vert w_{0}\right\Vert
	}+\sum_{j=1}^{\infty}\alpha_{j+1}e_{n_{j}}^{*}\right\Vert
	=\sum_{j=1}^{\infty}\left|\alpha_{j}\right|
	\]
	\[
	=\sum_{j=1}^{\infty}\left|\alpha_{j}\right|\left\Vert
	x_{n_{j}}^{*}\right\Vert \geq\left\Vert
	\sum_{j=1}^{\infty}\alpha_{j}x_{n_{j}}^{*}\right\Vert .
	\]
	Finally, it is easy to see that $T$ has no fixed point in
	$D$.   
\end{proof}

\begin{rem}
	Theorem \ref{thm sufficient} can be proved in a completely different way. Indeed, from the proof of Theorem 1 in \cite{Zippin2018}, we know that every $A(S)$ space contains an $\ell_1$-predual subspace that is isometric to a hyperplane in $c$ containing the point $(1,1,1,\dots)\in c$. It is easy to observe that such a hyperplane is bad. Therefore, by applying Theorem \ref{Theorem bad}, we conclude that $X^*$ lacks the $\sigma(X^*,X)$-FPP. This alternative proof relies on a deep technique developed by Zippin, whereas our approach is direct and self-contained. One may conjecture that each $A(S)$ space contains an isometric copy of the whole space $c$, which is the simplest example of bad hyperplane. However, it occurs that there exists an infinite-dimensional $A(S)$ space that does not contain an isometric copy of $c$ (see \cite{Casini-Miglierina-Piasecki-Vesely}).
\end{rem}

\begin{rem}\label{rem quotient}
	Let $X$ be a separable Banach space. Suppose that some infinite-dimensional $A(S)$ space is a subspace of a quotient $X/Y$ of $X$. Then Theorem \ref{thm sufficient} shows that $Y^{\bot}$ fails the $\sigma(Y^\bot,X/Y)$-FPP. It follows easily that also $X^*$ lacks the $\sigma(X^*,X)$-FPP.
\end{rem}

\section{Weak$^*$ fixed point property in the dual of separable Lindenstrauss space }

In this section we show that the sufficient condition stated in Remark \ref{rem quotient} is equivalent to the failure of the $\sigma(X^*,X)$-FPP whenever we consider a separable Lindenstrauss space $X$. Moreover, our result is linked to the characterization obtained in \cite{Casini-Miglierina-Piasecki2015}. Indeed, 
we applied the above mentioned notion of bad hyperplane  to state the following equivalence.
\begin{thm} [Theorem 4.1 in \cite{Casini-Miglierina-Piasecki2015}]
	Let $X$ be a predual of $\ell_1$. Then the following are equivalent:
	\begin{enumerate}
		\item $\ell_1$ lacks the $\sigma(\ell_1,X)$-FPP for nonexpansive mappings;
		
		\item\label{item quotient isometric} there is a quotient of $X$ isometric to a bad $W_f$;
		
		\item \label{item quotient contained}there is a quotient of $X$ that contains a subspace isometric to a bad $W_g$.		
	\end{enumerate}
\end{thm}

In the following result we replace bad $W_f$ by the space $A(S)$ of affine continuous functions on the Choquet simplex $S \!$. One can easily observe that there are bad $W_f$ which are not $A(S)$ spaces since their unit balls have no extreme points.

%
%
%
%

\begin{thm}\label{Characterization}	
	Let $X$ be a predual of $\ell_1$. The following statements are equivalent:
	\begin{enumerate}
		\item \label{item noFPP} $\ell_1$ lacks the $\sigma(\ell_1,X)$-FPP for nonexpansive mappings;
		
		\item \label{item affine}there is a quotient of $X$ isometric to some $A(S)$ space;
		
		\item \label{item containement}there is a quotient of $X$ containing an isometric copy of some $A(S)$ space.

	\end{enumerate}
\end{thm}

\begin{proof}
	We start by proving that (\ref{item noFPP}) implies (\ref{item affine}). From the implication (1) $\Rightarrow$ (4) of Theorem 4.1 in \cite{Casini-Miglierina-Piasecki2015}, we obtain that there is a subsequence $(e_{n_k}^*)_{k \in
		\mathbb{N}}$ of the standard basis $(e_{n}^*)_{n \in \mathbb{N}}$ in
	$\ell_1$ which is $\sigma(\ell_1,X)$-convergent to a norm-one
	element $e^* \in \ell _1$ with $e^*(n_k)\geq 0$ for all $k\in
	\mathbb{N}$.
	From the proof of the implication (4) $\Rightarrow$ (5) of Theorem 4.1 in \cite{Casini-Miglierina-Piasecki2015}, we know that $X$ has a quotient isometric to an $\ell_1$-predual hyperplane $W_f$ containing the point $(1,1,1,\dots)\in c$. By applying Corollary 2 in \cite{Japon-Prus2004} and ($\heartsuit$) one can prove that the positive face $S$ of the unit sphere of $\ell_1=W_f^*$ is $\sigma(\ell_1,W_f)$-compact and $W_f$ is isometric to $A(S)$.

	The implication (\ref{item affine}) $\Rightarrow$ (\ref{item containement}) is trivial.
	
	
	Finally, by applying Remark \ref{rem quotient}, we conclude that (\ref{item containement}) $\Rightarrow$ (\ref{item noFPP}). 
\end{proof}

The following example shows that the quotient in conditions (2) and (3) in Theorem \ref{Characterization} cannot be removed in a sense that these conditions can be replaced by: $X$ has a subspace isometric to an infinite-dimensional $A(S)$ space.	  

\begin{example}\label{Example Hyperplane}
	Let $f=(1/2,-1/4,1/8,-1/16,\dots) \in \ell_1$. Since $W_f$ is a bad hyperplane, $\ell_1$ fails the $\sigma(\ell_1,W_f)$-FPP. Moreover, $W_f$ does not have a quotient containing an isometric copy of $c$ (see Example 2.4 in \cite{Casini-Miglierina-Piasecki2015}).

	We claim that the hyperplane $W_f$ does not contain any infinite-dimensional $A(S)$ space. By contradiction, suppose that $A(S) \subset W_f$. Let $\left\lbrace e^*_n \right\rbrace$ be the standard basis in $\ell_1=A(S)^*$. Since $A(S) \subset \mathcal{C}(S)$, by Lemma \ref{lemmaDS}, there exists a sequence of signs $(\varepsilon(n))_{n\in \mathbb{N}}$, $\varepsilon(n) = \pm 1$ for all $n \in \mathbb{N}$, such that $e^*_n (\mathbbm{1})=\varepsilon(n)$, where $\mathbbm{1}$ denotes the constant function equal to $1$ on $S$. Let $\widetilde{e^*_n}$ denote the norm-preserving extension of $e^*_n$ to the whole $W_f$. Then
	$$
	2=\left\| e^*_n \pm e^*_m\right\| \leq \left\| \widetilde{e^*_n} \pm \widetilde{e^*_m}\right\|\leq 2.
	$$
	These relations mean that $\left\lbrace \widetilde{e^*_n}\right\rbrace $ is represented in $\ell_1=W_f^*\,$ by a sequence of disjoint blocks of norm $1$. Moreover, for every $n \in \mathbb{N}$, it holds $\widetilde{e^*_n}(\mathbbm{1})=e^*_n(\mathbbm{1})=\varepsilon(n)$. This shows that $\mathbbm{1}$ is represented in $W_f$ by $x=\left(x(1),x(2), \dots \right) \in S_{W_f}$ such that for every $n \in \mathbb{N}$ we have
	$$
	x(i)=\textrm{sgn}\, \widetilde{e^*_n}(i)\quad 
	$$
	if $i \in \textrm{supp}\,\widetilde{e^*_n} :=\left\lbrace i \in \mathbb{N}:\widetilde{e^*_n}(i) \neq 0 \right\rbrace$. Since $x\in B_c$, we have $\lim_{n \rightarrow \infty}x(n)=1$ or $\lim_{n \rightarrow \infty}x(n)=-1$. However, there is no such $x\in W_f$.

\end{example}

Theorem \ref{Characterization} can be easily extended from the case of $\ell_1$-preduals to the whole class of separable $L_1$-preduals. 

\begin{thm} \label{thm separable Lindenstrauss} Let $X$ be a separable Lindenstrauss space. The following statements are equivalent:
	\begin{enumerate}
		\item $X^*$ lacks the $\sigma(X^*,X)$-FPP for nonexpansive mappings;
		
		\item there is a quotient of $X$ isometric to some $A(S)$ space;
		
		\item there is a quotient of $X$ containing an isometric copy of some $A(S)$ space.

	\end{enumerate}
\end{thm}
\begin{proof}
	By taking into the account Theorem \ref{Characterization}, it is enough to consider the case when $X^*$ is nonseparable. Theorem 2.3 in \cite{Lazar-Lindenstrauss1971} states that a separable Lindenstrauss space $X$ with nonseparable dual contains a subspace isometric to the space $\mathcal{C}(\Delta)$, where $\Delta$ is the Cantor set. Since $\mathcal{C}(\Delta)$ contains an isometric copy of $c$, by Proposition 3.1 in \cite{Casini-Miglierina-Piasecki2015} there is a $1$-complemented copy of $c$. Therefore, $X$ has a quotient isometric to $c$. This shows that (2) and (3) hold true. Finally, by applying Corollary 3.4 in \cite{Casini-Miglierina-Piasecki2015} we conclude the proof. 
\end{proof}

\begin{rem}\label{Remark C(K)}
	In Theorem \ref{Characterization} the space $A(S)$ cannot be replaced by any space $\mathcal{C}(K)$ of continuous functions on the compact Hausdorff set $K$. Indeed, if $K$ is finite, then $\mathcal{C}(K)=\ell_{\infty}^{(n)}$ for some $n \in \mathbb{N}$. By \cite{Lazar-Lindestrauss1966,Michael-Pelczynski}, we know that every separable Lindenstrauss space contains an isometric copy of $\ell_{\infty}^{(n)}$ for every $n$. Since $\ell_{\infty}^{(n)}$ is always $1$-complemented,  $X$ has a quotient isometric to $\ell_{\infty}^{(n)}$. Moreover, if $K$ is an  infinite countable set, then, by Mazurkiewicz-Sierpi\'{n}ski Theorem (\cite{Mazurkiewicz-Sierpinski}), we know that $\mathcal{C}(K)$ contains an isometric copy of $c$. However, by Example \ref{Example Hyperplane}, we know that there is an $\ell_1$-predual such that $\ell_1$ fails the weak$^*$-FPP, whereas it does not have a quotient containing an isometric copy of $c$. It remains to consider the case where $K$ is uncountable. However, under this assumption $\mathcal{C}(K)^*$ is nonseparable and therefore $X$ cannot be an $\ell_1$-predual. 
\end{rem}

We conclude our paper by pointing out that some other equivalent conditions for the weak$^*$-FPP are known in the literature. We refer the interested reader to \cite{Casini-Miglierina-Piasecki2015,Casini-Miglierina-Piasecki-Popescu2017,Casini-Miglierina-Piasecki-Popescu2018,Piasecki1,Piasecki2}.

\end{document}